\documentclass[12pt]{amsart}
\usepackage{amssymb, latexsym}

\theoremstyle{plain}
\newtheorem{theorem}{Theorem}[section]
\newtheorem{corollary}[theorem]{Corollary}
\newtheorem{proposition}[theorem]{Proposition}
\newtheorem{lemma}[theorem]{Lemma}

\theoremstyle{definition}
\newtheorem{definition}[theorem]{Definition}
\newtheorem{remark}[theorem]{Remark}

\numberwithin{equation}{section}

\usepackage{tikz,tikz-cd}
\usetikzlibrary{arrows,snakes,backgrounds,calc}
\usetikzlibrary{decorations.pathreplacing,angles,quotes}

\tikzstyle{qnode}=[circle,draw=black,thick, inner sep=1pt]
\tikzstyle{knode}=[circle,draw=black,thick, text width = 12 pt, align=center, inner sep=0pt]
\tikzstyle{hnode}=[circle,text width = 15 pt,align=center,inner sep=1pt]

\usepackage{graphicx}

\usepackage{hyperref}

\input xy
\xyoption{all}

\begin{document}

\title[Quasigroup words and reversible automata]
{Quasigroup words and reversible automata}
\author[J.D.H.~Smith]{Jonathan D.H.~Smith$^1$}
\author[S.G.~Wang]{Stefanie G.~Wang$^2$}
\address{$^1$ Dept. of Mathematics\\
Iowa State University\\
Ames, Iowa 50011, U.S.A.}
\address{$^2$ Dept. of Mathematics \& Statistics\\
Smith College\\
Northampton, Massachusetts 01063, U.S.A.}
\email{$^1$jdhsmith@iastate.edu\phantom{,}}
\email{$^2$stwang@smith.edu}


\begin{abstract}
This paper examines two related topics: the linearization of the reversible automata of Gvaramiya and Plotkin, and the problem of finding a faithful representation of the words in a central quasigroup that respects the triality symmetry of the language of quasigroups.
\end{abstract}

\maketitle


\section{Dedication} This paper is written as the first author's homage to his long-time friend Ivo Rosenberg, whom he met initially at the 25-th \textit{Arbeitstagung Allgemeine Algebra}, held in May, 1978 at the University of Klagenfurt in the Austrian province of Carinthia. At that time, universal algebra had fallen under the cloud of Graham Higman's review of Paul Cohn's text, with its notorious comment on the subject that ``nobody should specialize in it'' \cite{Higman}. In this climate, encountering Ivo's deep and powerful masterpiece \cite{Rosenberg} came as a ray of brilliant sunshine, placing the duality between relations and clones firmly in the long tradition of Felix Klein's view of geometry as a duality between relations and their automorphism groups. Even Higman worked this vein \cite{Higman2}.

Ivo Rosenberg was an inspiring mathematician, not only in his work, but also in his gentle and open personality. The first author was fortunate to receive Ivo's hospitality at a time when he was in visa limbo, waiting in Montreal to  start a tenure-track job at Iowa State University. Ivo had made his own Atlantic crossing when Warsaw Pact forces invaded Czechoslovakia after the Prague Spring. Working in the Sudan at the time, he was able to move to Canada to avoid the repression sweeping through his native land. After settling in Montreal, he was involved with a number of important conferences and workshops that brought together mathematicians working in logic, algebra, combinatorics, and computer science. The meetings had a lasting influence on all these areas, an influence willingly accepted and gratefully acknowledged by the authors of this paper.

\section{Introduction}

We examine two apparently separate topics that turn out to be related: the linearization of the reversible automata of Gvaramiya and Plotkin, and the problem of finding a faithful representation of the words in a central (or linear) quasigroup that behaves well with respect to the triality symmetry ($S_3$-action or parastrophy) of the language of quasigroups.

Let $\mathbf{Qtp}$ be the category of quasigroup homotopies. There have been two approaches to the problem of realizing homotopies between quasigroups as homomorphisms of algebraic structure. The first approach, which is due to Gvaramiya and Plotkin \cite{Gv,GvPl}, used heterogeneous algebras implementing \emph{reversible automata} (Section~\ref{S:revrsaut}). The second approach establishes a functor $\Delta\colon\mathbf{Qtp}\to\mathbf P$, which is known as \emph{semisymmetrization}, to the category $\mathbf P$ of homomorphisms between semisymmetric quasigroups, quasigroups that satisfy the identity $x(yx)=y$ (Section~\ref{S:Backgrnd}). The image of the functor $\Delta$ consists of semisymmetric quasigroups with an extra structure, that of a \emph{semisymmetrized algebra} (Section~\ref{S:SeSymAlg}). Theorem~\ref{T:LinSemiS} describes the structure of linear semisymmetrized algebras. Theorem~\ref{T:LiSSARvA} then identifies the linear reversible automaton that underlies a linear semisymmetrized algebra.

An algebra $Q$ is called \emph{central} if the diagonal $\widehat{Q}=\{(x,x)\mid x\in Q\}$ is a normal subalgebra (congruence class) in the direct square $Q\times Q$. (This is the terminology of \cite{MV}, not to be confused with the central relations of \cite{Rosenberg}.) Following our enumeration of words in free quasigroups \cite{JSW}, we have now begun to consider the enumeration of words in free central quasigroups. These enumerations are critically dependent on the triality symmetry of the language of quasigroups (Section~\ref{S:triality}), which relates the various conjugates or parastrophes of a quasigroup. However, the usual linear representations of central quasigroups (\cite[\S3.2]{RTIGFQ},  \cite[Prop.~11.1]{IQTR}, cf. Proposition~\ref{P:CePiFrGp}) do not behave nicely with respect to the triality symmetry. This has presented an obstacle to our attempts at enumeration of free central quasigroup words.

The linear reversible automata of Theorem~\ref{T:LiSSARvA}(a) come to the rescue, as their operations \eqref{E:ratlinSS} suggest alternative linear expressions for the central quasigroup operations, embodied in Theorem~\ref{T:CePiFrGp}. These expressions do transform homogeneously under the triality symmetry. An added benefit is that the integer coefficients of the polynomial representations of free central quasigroup words are all nonnegative, which is certainly not the case for the traditional representations from Proposition~\ref{P:CePiFrGp}.

Imposing centrality leads to equivalences, such as 
$$
((a_0/a_1).a_2a_3)/(a_4\backslash a_0)
=(a_4.a_2a_3)/a_1
$$
\eqref{E:TwoWords}, between reduced words that have different lengths in an absolutely free quasigroup. The new representation furnished by Theorem~\ref{T:CePiFrGp} provides an easy way to identify when such unexpected cancellations occur, in the form of the argument elimination patterns presented in \S\ref{SS:ArgtElim} and Figure~\ref{F:PatOfCol}.

In \S\ref{SS:16shrtst}, a particular $2$-dimensional representation from Theorem~\ref{T:CePiFrGp} is used to provide a geometric display (Figure~\ref{F:16shrtst}) of the sixteen shortest words in the free central quasigroup on a singleton $\{a\}$. Exact triality symmetry between these words translates to an approximate symmetry in the positions of their display points.

The paper generally follows the notation and conventions of \cite{SRPMA}. In particular, we default to algebraic notation with functions and functors following their arguments, either on the line or as a superfix.

\section{Quasigroups, homotopy, and semisymmetry}\label{S:Backgrnd}

\subsection{Quasigroups}\label{SS:quasigps}

An \emph{equational quasigroup} $(Q, \cdot,/,\backslash)$ is a set equipped with three \emph{basic} binary operations, multiplication $\cdot$, right division /, and left division $\backslash$ such that for all $x, y\in Q$, the following identities are satisfied:
\begin{equation}\label{E:QgpIdens}
 \begin{array}{lll}
( \mbox{SL} ) & x\cdot(x\backslash y) = y \, ;\qquad
( \mbox{SR} ) & y = (y/x)\cdot x \, ; \\
( \mbox{IL} ) & x\backslash (x\cdot y) = y \, ;\qquad
( \mbox{IR} ) & y = (y\cdot x)/x \, .
 \end{array}
\end{equation}
Note that (IL), (IR) give the respective injectivity of the left multiplication
$$
L(x)\colon Q\to Q;y\mapsto xy
$$
and right multiplication
$$
R(x)\colon Q\to Q;y\mapsto yx
$$
for $x\in Q$, while (SL), (SR) give their surjectivity. Thus an equational quasigroup $(Q,\cdot,/,\backslash)$ yields a combinatorial quasigroup $(Q,\cdot)$. Conversely, a combinatorial quasigroup $(Q,\cdot)$ yields an equational quasigroup $(Q,\cdot,/,\backslash)$ with $x/y=xR(y)^{-1}$ and $x\backslash y=yL(x)^{-1}$.

In an equational quasigroup $(Q,\cdot,/,\backslash)$, the three equations
\begin{equation}\label{E:firsthre}
x_1\cdot x_2=x_3\, ,\qquad x_3/x_2=x_1\, ,\qquad x_1\backslash x_3=x_2
\end{equation}
involving the basic operations are equivalent. Now consider the \emph{opposite} operations
$$
x\circ y=y\cdot x\, ,\qquad x/\negthinspace/y=y/x\, ,\qquad x\backslash\negthinspace\backslash y=y\backslash x
$$
on $Q$. Then the equations \eqref{E:firsthre} are further equivalent to the equations
$$
x_2\circ x_1=x_3\, ,\qquad x_2/\negthinspace/x_3=x_1\, ,\qquad x_3\backslash\negthinspace\backslash x_1=x_2\, .
$$
Thus each of the basic and opposite operations
\begin{equation}\label{E:conjugac}
(Q,\cdot),\quad(Q,/),\quad(Q,\backslash),
\quad(Q,\circ),\quad(Q,/\negthinspace/),\quad(Q,\backslash\negthinspace\backslash)
\end{equation}
forms a (combinatorial) quasigroup. In particular, note that the identities (IR) in $(Q,\backslash)$ and (IL) in (Q,/) yield the respective identities
$$
 \begin{array}{ll}
( \mbox{DL} ) & x/(y\backslash x)=y \, ,\index{DL@(DL)}\\
( \mbox{DR} ) & y=(x/y)\backslash x \index{DR@(DR)}
 \end{array}
$$
in the basic quasigroup divisions. The six quasigroups \eqref{E:conjugac} are known as the \emph{conjugates}, ``parastrophes'' \cite{Sade57a} or ``derived quasigroups'' \cite{James} of $(Q,\cdot)$.

\subsection{The homotopy category}

A \emph{homotopy} $(f_1,f_2,f_3)\colon Q\to Q'$ from a quasigroup $Q$ to a quasigroup $Q'$ is a triple of functions from $Q$ to $Q'$ such that
$$
xf_1\cdot yf_2=(x\cdot y)f_3
$$
for all $x,y$ in $Q$. Write $\mathbf Q$ for the category of homomorphisms between quasigroups, and $\mathbf{Qtp}$ for the category of homotopies between quasigroups. Then there is a forgetful functor\index{S i@$\Sigma$ (homomorphism as homotopy)}
\begin{equation}\label{E:qtoqtp}
\Sigma:\textbf{Q}\rightarrow\textbf{Qtp}
\end{equation}
preserving objects, sending a quasigroup homomorphism $f:Q\rightarrow Q'$ to the homotopy $(f,f,f):Q\rightarrow Q'$. A function $f:Q\rightarrow Q'$ connecting the underlying sets of equational quasigroups $(Q,\cdot,/,\backslash)$ and $(Q',\cdot,/,\backslash)$ is a quasigroup homomorphism if it is a homomorphism $f:(Q,\cdot)\rightarrow(Q',\cdot)$ for the multiplications. Thus a homotopy $(f_1,f_2,f_3)$ having equal components $f_1=f_2=f_3$ is an element of the image of the morphism part of the forgetful functor \eqref{E:qtoqtp}.

\subsection{Semisymmetric quasigroups}

A quasigroup is \emph{semisymmetric} if it satisfies the identity $x\cdot yx=y$. (Compare \cite[Ex. 9]{Sm130} for an interpretation of this identity in terms of the semantic triality of quasigroups.) Using right and left multiplications, semisymmetry amounts to the equality
\begin{equation}\label{E:R(x)L(x)}
R(x)=L(x)^{-1}\,,
\end{equation}
and may thus be expressed in equivalent form as $xy\cdot x=y$. The simplest models of semisymmetric quasigroups are abelian groups with $-x-y$ as the multiplication operation.

\subsection{Semisymmetrization}

Let $\mathbf{P}$ denote the category of homomorphisms between semisymmetric quasigroups. Then each quasigroup $Q$ or $(Q,\cdot,/,\backslash)$ defines a semisymmetric quasigroup structure $Q\Delta$ on the direct cube $Q^3$ with multiplication as follows:
\begin{equation}\label{E:semisym}
\begin{matrix}
\phantom{-}\\
(\phantom{m} x_1\phantom{n},\phantom{m} x_2\phantom{m},\phantom{m} x_3\phantom{m})\phantom{l}\cdot\phantom{l}\\
\phantom{-}\\
(\phantom{m} y_1\phantom{m},\phantom{m} y_2\phantom{m},\phantom{m} y_3\phantom{m})=\\
\phantom{-}\\
(x_2/\negthinspace/y_3,x_3\backslash\negthinspace\backslash y_1,x_1\cdot y_2)\phantom{|.}\\
\phantom{-}\\
\end{matrix}
\end{equation}
\cite{Sm63}. If $(f_1,f_2,f_3):(Q,\cdot)\rightarrow(Q',\cdot)$ is a quasigroup homotopy, define
\begin{equation}\label{E:deltaonf}
(f_1,f_2,f_3)^{\Delta}:Q\Delta\rightarrow Q'\Delta;(x_1,x_2,x_3)\mapsto(x_1f_1,x_2f_2,x_3f_3).
\end{equation}
This map is a quasigroup homomorphism. Indeed, for $(x_1,x_2,x_3)$ and $(y_1,y_2,y_3)$ in $Q\Delta$, one has
\begin{alignat*}{1}
(x_1&f_1,x_2f_2,x_3f_3)\cdot(y_1f_1,y_2f_2,y_3f_3)\\
&=(x_2f_2/\negthinspace/y_3f_3,x_3f_3\backslash\negthinspace\backslash y_1f_1,x_1f_1\cdot y_2f_2)\\
&=\big((x_2/\negthinspace/y_3)f_1,(x_3\backslash\negthinspace\backslash y_1)f_2,(x_1\cdot y_2)f_3\big)\\
&=\big((x_1,x_2,x_3)\cdot(y_1,y_2,y_3)\big)(f_1,f_2,f_3)^{\Delta},
\end{alignat*}

Consider the functor
\begin{equation}\index{D a@$\Delta$ (semisymmetrization)}
\Delta:\mathbf{Qtp}\rightarrow\mathbf{P},
\end{equation}
known as the \textit{semisymmetrization functor}\index{semisymmetrization functor}, which has object part \eqref{E:semisym} and morphism part \eqref{E:deltaonf}. This functor has a left adjoint, namely the restriction $\Sigma:\mathbf{P}\rightarrow\mathbf{Qtp}$ of the forgetful functor \eqref{E:qtoqtp} \cite[Th. 5.2]{Sm63}. The unit of the adjunction at a semisymmetric quasigroup $P$ is the homomorphism
\begin{equation}\label{E:unit}
\eta_P:P\to P\Sigma\Delta;x\mapsto (x,x,x)
\end{equation}
\cite[(5.3)]{Sm63}. The counit $\varepsilon_Q$ at a quasigroup $Q$ is the homotopy
\begin{equation}\label{E:counit}
(\pi_1,\pi_2,\pi_3):Q\Delta\Sigma\to Q
\end{equation}
with $(x_1,x_2,x_3)\pi_i=x_i$ for $1\le i\le 3$ \cite[(5.4)]{Sm63}.

\section{Triality}\label{S:triality}

\subsection{Triality action of the symmetric group $S_3$}

The symmetric group $S_3$ on the $3$-element set $\{1,2,3\}$ is presented as
$$
\big\langle \sigma,\tau \,\big\vert\, \sigma^2=\tau^2=(\sigma\tau)^3=1 \big\rangle
$$
writing $\sigma$ and $\tau$ for the respective transpositions $(12)$ and $(23)$. The Cayley diagram of the presentation is
\begin{equation}\label{E:cayls3st}
\begin{matrix}
1  &\Longleftrightarrow   &\tau   &\longleftrightarrow     &\tau\sigma  \\
\updownarrow   &            &       &                   &\Updownarrow   \\
\sigma &\Longleftrightarrow   &\sigma\tau  &\longleftrightarrow     &\sigma\tau\sigma
\end{matrix}
\end{equation}
with $\leftrightarrow$ for right multiplication by $\sigma$ and $\Leftrightarrow$ for right multiplication by $\tau$.

Now consider the full set
\begin{equation}\label{E:conjugat}
\{\cdot, \backslash, /\negthinspace/, /, \backslash\negthinspace\backslash, \circ\}
\end{equation}
of all quasigroup operations, both basic and opposite. For current purposes, it will be convenient to use postfix notation for binary operations, setting $x\cdot y=xy\,\mu$ and rewriting the first equation of \eqref{E:firsthre} in the form
\begin{equation}\label{E:x1x2mux3}
x_1x_2\,\mu =x_3 \, .
\end{equation}
The full set \eqref{E:conjugat} is construed as the homogeneous space
\begin{equation}\label{E:homospop}
\mu^{S_3}=\{\mu^g\mid g\in S_3\}
\end{equation}
for a regular right permutation action of the symmetric group $S_3$, such that \eqref{E:x1x2mux3} is equivalent to
\begin{equation}\label{E:x1gx2gmg}
x_{1g}x_{2g}\mu^g=x_{3g}
\end{equation}
for each $g$ in $S_3$. This action is known as \emph{triality}.\footnote{Triality in this sense is not to be confused with the distinct but related notion that ultimately arises from the action of $S_3$ on the Dynkin diagram $D_4$ (compare, say, \cite{HaNa}).} The six binary operations, in their positions corresponding to the Cayley diagram \eqref{E:cayls3st}, are displayed in Figure~\ref{F:trialops}.
\vskip .2in
\begin{figure}[bht]
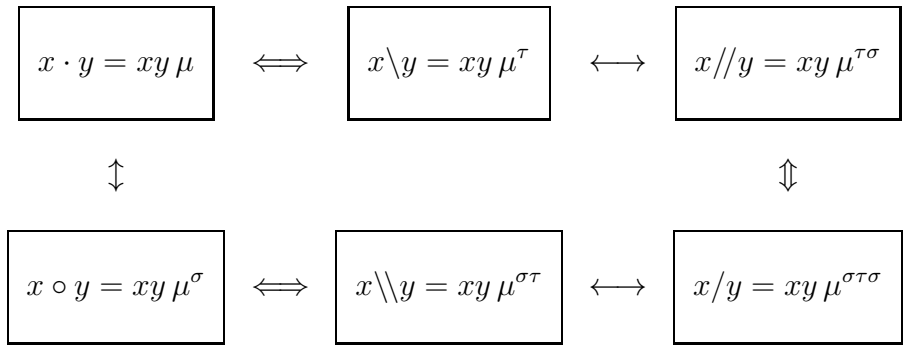

$$
\begin{matrix}
\boxed{\phantom{\Bigg|}x\cdot y=xy\,\mu\phantom{\Bigg|}}
&\Longleftrightarrow
&\boxed{\phantom{\Bigg|}x\backslash y=xy\,\mu^{\tau}\phantom{\Bigg|}}
&\longleftrightarrow
&\boxed{\phantom{\Bigg|}x/\negthinspace/y=xy\,\mu^{\tau\sigma}\phantom{\Bigg|}}
\\
&            &       &                   & \\
\updownarrow   &            &       &                   &\Updownarrow   \\
&            &       &                   & \\
\boxed{\phantom{\Bigg|}x\circ y=xy\,\mu^{\sigma}\phantom{\Bigg|}}
&\Longleftrightarrow
&\boxed{\phantom{\Bigg|}x\backslash\negthinspace\backslash y=xy\,\mu^{\sigma\tau}\phantom{\Bigg|}}
&\longleftrightarrow
&\boxed{\phantom{\Bigg|}x/y=xy\,\mu^{\sigma\tau\sigma}\phantom{\Bigg|}}
\end{matrix}
$$
\vskip .3in
\caption{Symmetry of the quasigroup operations.
}\label{F:trialops}
\end{figure}
In the figure, the opposite of each operation $\mu^g$ is given by $\mu^{\sigma g}$. It follows that passage to the opposite operation corresponds to left multiplication by the transposition $\sigma$ in the symmetric group $S_3$. The three pairs of opposite operations lie in the respective columns of Figure~\ref{F:trialops}.

\subsection{Invertible elements of the monoid of binary words}

Just as the left multiplication by $\sigma$ within the symmetric group $S_3$ allows for a simple interpretation, so does the left multiplication by $\tau$. Let $M$ be the complete set of all derived binary operations on a quasigroup. Taken together, these operations constitute the free algebra on two generators $x$, $y$ in the variety $\mathbf{Q}$ of quasigroups. A multiplication $*$ is defined on $M$ by
\begin{equation}\label{E:mltonbin}
xy(\alpha*\beta)=x\,xy\alpha\,\beta\, .
\end{equation}
The right projection $xy\epsilon=y$ also furnishes a binary operation $\epsilon$.

\begin{lemma}
The set $M$ of all derived binary quasigroup operations forms a monoid $(M,*,\epsilon)$ under the multiplication \eqref{E:mltonbin}, with identity element $\epsilon$.
\end{lemma}

\begin{proof}
Observe that
$$
xy(\alpha*\epsilon)=x\,xy\alpha\,\epsilon=xy\alpha
\quad
\mbox{ and }
\quad
xy(\epsilon*\alpha)=x\,xy\epsilon\,\alpha=xy\alpha
$$
for $\alpha$ in $M$, so $\epsilon$ is an identity element. Consider $\alpha$, $\beta$, $\gamma$ in $M$. Then
\begin{align*}
xy\big((\alpha*\beta)*\gamma\big)&=x\,xy(\alpha*\beta)\,\gamma
\\
&
=xxxy\alpha\beta\gamma
\\
&
=x\,xy\alpha(\beta*\gamma)=xy\big(\alpha*(\beta*\gamma)\big) \, ,
\end{align*}
confirming the associativity of the multiplication \eqref{E:mltonbin}.
\end{proof}

The significance of the left multiplication by $\tau$ then follows.

\begin{proposition}\label{T:hyperqgp}
For each element $g$ of $S_3$, the binary operation $\mu^g$ is a unit of the monoid $M$, with inverse $\mu^{\tau g}$.
\end{proposition}

\begin{proof}
The (IL) identity $x\backslash(x\cdot y)=y$ becomes $x\,xy\mu\,\mu^{\tau}=y$ or $\mu*\mu^{\tau}=\epsilon$. Similarly the (SL) identity $x\cdot(x\backslash y)=y$ becomes $x\,xy\mu^{\tau}\,\mu=x$ or $\mu^{\tau}*\mu=\epsilon$. This means that $\mu$ and $\mu^{\tau}$ are mutual inverses.

The (IR) identity $(y\cdot x)/x=y$ is $x/\negthinspace/(x\circ y)=y$. This becomes $x\,xy\mu^{\sigma}\,\mu^{\tau\sigma}=y$ or $\mu^{\sigma}*\mu^{\tau\sigma}=\epsilon$. Similarly (SR), namely $(y/ x)\cdot x=y$, may be written as $x\circ(x/\negthinspace/ y)=y$. This becomes $x\,xy\mu^{\tau\sigma}\,\mu^{\sigma}=y$ or $\mu^{\tau\sigma}*\mu^{\sigma}=\epsilon$. Thus $\mu^{\sigma}$ and $\mu^{\tau\sigma}$ are mutual inverses.

The (DR) identity $(x/y)\backslash x=y$ is $x\backslash\negthinspace\backslash (x/y)=y$, which translates to $x\,xy\mu^{\tau\sigma\tau}\,\mu^{\sigma\tau}=y$ or $\mu^{\tau\sigma\tau}*\mu^{\sigma\tau}=\epsilon$. Finally, the (DL) identity $x/(y\backslash x)=y$ is $x/(x\backslash\negthinspace\backslash y)=y$, which becomes $x\,xy\mu^{\sigma\tau}\,\mu^{\tau\sigma\tau}=y$ or $\mu^{\sigma\tau}*\mu^{\tau\sigma\tau}=\epsilon$. Thus $\mu^{\sigma\tau}$ and $\mu^{\tau\sigma\tau}$ are mutual inverses.
\end{proof}

\begin{corollary}\label{C:hyperqgp}
The quasigroup identities $(\mathrm{SL}), (\mathrm{IL}), (\mathrm{SR}), (\mathrm{IR}), (\mathrm{DL}), (\mathrm{DR})$ of \S\ref{SS:quasigps} all take the form
\begin{equation}\label{E:hypercan}
x\,xy\mu^{\tau g}\,\mu^{g}=y
\end{equation}
for an element $g$ of $S_3$.
\end{corollary}

\section{Reversible automata}\label{S:revrsaut}

\subsection{Heterogeneous algebras}

Gvaramiya and Plotkin \cite{Gv} \cite{GvPl} introduced certain heterogeneous algebras for the reduction of quasigroup homotopies to homomorphisms. A \emph{reversible automaton} (\emph{of quasigroup type}) is a triple $(S_1,S_2,S_3)$ of sets or state spaces $S_i$, equipped with respective operations
\begin{align*}
&\mu:S_1\times S_2\to S_3;(x_1,x_2)\mapsto x_1\cdot x_2 \, ,\\
&\rho:S_3\times S_2\to S_1;(x_3,x_2)\mapsto x_3/x_2 \, ,\\
&\lambda:S_1\times S_3\to S_2;(x_1,x_3)\mapsto x_1\backslash x_3 \,
\end{align*}
of \emph{multiplication}, \emph{right division} and \emph{left division} satisfying the identities:
\begin{equation}\label{E:revautid}
\begin{cases}
( \mbox{ILA} ) & x_1\backslash (x_1\cdot x_2) = x_2 ;\\
( \mbox{IRA} ) & x_1 = (x_1\cdot x_2)/x_2 ;\\
( \mbox{SLA} ) & x_1\cdot(x_1\backslash x_3) = x_3 ;\\
( \mbox{SRA} ) & x_3 = (x_3/x_2)\cdot x_2.
\end{cases}
\end{equation}
analogous to \eqref{E:QgpIdens}. If $(S_1,S_2,S_3)$ and $(S'_1,S'_2,S'_3)$ are reversible automata, a homomorphism (in the sense of heterogeneous algebras \cite{BL,GM,Hi,Lug}) is a triple $(f_1,f_2,f_3)$ of maps $f_i:S_i\to S'_i$ (for $1\le i\le 3$) such that
\begin{equation}\label{E:rerautho}
x_1^{f_1}\cdot x_2^{f_2}=(x_1\cdot x_2)^{f_3}\, ,\ x_3^{f_3} / x_2^{f_2}=(x_3/x_2)^{f_1}\, ,\ x_1^{f_1}\backslash x_3^{f_3}=(x_1\backslash x_3)^{f_2}
\end{equation}
for $x_i$ in $S_i$. With homomorphisms as morphisms, the class of reversible automata forms a category $\mathbf{RAt}$, a variety of heterogeneous algebras.

Each quasigroup $Q$ yields a reversible automaton $Q^{\mathsf{at}}$ or $(Q,Q,Q)$ with equal state spaces. From \eqref{E:rerautho}, it then follows that a quasigroup homotopy $f=(f_1,f_2,f_3):Q\to Q'$ will produce a corresponding homomorphism $f^{\mathsf{at}}:Q^{\mathsf{at}}\to Q'^{\mathsf{at}}$ of reversible automata. Thus a functor $\mathsf{at}:\mathbf{Qtp}\to\mathbf{RAt}$ is defined, corestricting to an equivalence $\mathsf{at}:\mathbf{Qtp}\to\mathbf{QAt}$ of the homotopy category $\mathbf{Qtp}$ with a category $\mathbf{QAt}$ of homomorphisms between reversible automata.

\subsection{Pure reversible automata}\label{S:regreaut}

For any set $S$ (empty or not), there are reversible automata
\begin{equation}\label{E:iregular}
(\varnothing,\varnothing,S)\, ,\ (S,\varnothing,\varnothing)\, ,\ (\varnothing,S,\varnothing)
\end{equation}
of quasigroup type in which the operations of multiplication, right division, and left division respectively embed the empty set in $S$, while the other two operations in each are the identity $1_{\varnothing}$ on the empty set. The remaining possibilities are described by the following.

\begin{proposition}\label{P:regreaut}
Let $(S_1,S_2,S_3)$ be a reversible automaton of quasigroup type, with at most one empty state space. Then the state spaces $S_1$, $S_2$ and $S_3$ are isomorphic, nonempty sets.
\end{proposition}

\begin{proof}
If $S_2$ were empty, the existence of the left division (as a map from $S_1\times S_3$ to $S_2$) would imply the emptiness of at least one of $S_1$ and $S_3$. Thus $S_2$ is nonempty. Fix an element $s_2$ of $S_2$. Then by (IRA) and (SRA), the maps
$$
S_1\to S_3;x_1\mapsto x_1\cdot s_2
\qquad\mbox{
and
}\qquad
S_3\to S_1;x_3\mapsto x_3/s_2
$$
are mutually inverse, showing that $S_1$ and $S_3$ are isomorphic. In particular, they are both nonempty. For an element $s_1$ of $S_1$, the maps
$$
S_2\to S_3;x_2\mapsto s_1\cdot x_2
\qquad\mbox{
and
}\qquad
S_3\to S_2;x_3\mapsto s_1\backslash x_3
$$
are mutually inverse by (ILA) and (SLA). Thus the state spaces $S_3$ and $S_2$ are isomorphic.
\end{proof}

\begin{definition}\label{D:regreaut}
A reversible automaton of quasigroup type is said to be \emph{pure} if its three state spaces are isomorphic sets.
\end{definition}


Now consider a pure reversible automaton $(S_1,S_2,S_3)$. By the purity, there is a set $Q$ with isomorphisms
\begin{equation}\label{E:isoq2aut}
l_i:Q\to S_i
\end{equation}
for $1\le i\le 3$.

\begin{lemma}\label{L:ReAt2Qgp}
Define a respective multiplication, right, and left division on $Q$ by
\begin{align}
x\cdot y&=(x^{l_1}\cdot y^{l_2})l_3^{-1}\, ,\label{E:dotonq} \\
x/y&=(x^{l_3}/y^{l_2})l_1^{-1}\, ,\label{E:slotonq} \\
x\backslash y&=(x^{l_1}\backslash y^{l_3})l_2^{-1}\, .\label{E:bkslotoq}
\end{align}
Then $(Q,\cdot,/,\backslash)$ is a quasigroup.
\end{lemma}

\begin{proof}
The identities \eqref{E:revautid} on $(S_1,S_2,S_3)$ yield the identities \eqref{E:QgpIdens} on $Q$, making $Q$ a quasigroup. 
\end{proof}

The definition of the operations on $Q$ shows that $$(l_1,l_2,l_3):Q^{\mathsf{at}}\to(S_1,S_2,S_3)$$ is an isomorphism in $\mathbf{RAt}$. If $(S_1,S_2,S_3)$ is isomorphic to $Q'^{\mathsf{at}}$ for a quasigroup $Q'$, then the isomorphism between $Q^{\mathsf{at}}$ and $Q'^{\mathsf{at}}$ shows that $Q$ and $Q'$ are isotopic. Summarizing, one has the following result, originally stated without the explicit purity hypothesis as \cite[Th. 1(1)]{Gv} \cite[Th. 1]{GvPl}\footnote{These papers used an implicit assumption of nonemptiness.}.

\begin{theorem}\label{T:rvat2qgp}
Let $R$ be a pure reversible automaton of quasigroup type.
\begin{itemize}
\item[$(\mathrm{a})$]
Within the category $\mathbf{RAt}$, the automaton $R$ is isomorphic to an automaton of the form $Q^{\mathsf{at}}$ for a quasigroup $Q$.
\item[$(\mathrm{b})$]
The quasigroup $Q$ of $(\mathrm{a})$ is unique up to isotopy.
\end{itemize}
\end{theorem}

\section{Semisymmetrized algebras and their linearization}\label{S:SeSymAlg}

\subsection{Diagonal algebras}

Suppose that an algebra $(P,\alpha)$ with a ternary operation $\alpha$ satisfies the \emph{diagonal identity}
\begin{align}\label{E:assocmu}
\big((x_{11},x_{12},x_{13})^{\alpha},
(x_{21},x_{22},x_{23})^{\alpha},
(x_{31},&x_{32},x_{33})^{\alpha}\big){\alpha}\\
&=(x_{11},x_{22},x_{33}){\alpha}\notag
\end{align}
for $x_{ij}$ in $P$ with $1\le i,j\le 3$, and the idempotence
\begin{equation}\label{E:unitmu}
(x,x,x)\alpha=x
\end{equation}
for $x$ in $P$. Thus $(P,\alpha)$ forms a \emph{diagonal algebra} in the sense of P\l onka \cite{Pl1} \cite{Pl2} \cite[Ex. 5.2.2]{RS}.

\subsection{Semisymmetrized algebras}

\begin{definition}\label{D:semisymd}\cite[Def'n.~34]{SmQHSRA}
An algebra $(P,\cdot,\alpha)$ equipped with a binary multiplication denoted by $\cdot$ or juxtaposition, and an idempotent ternary operation $\alpha$, is a \emph{semisymmetrized algebra} if:
\begin{enumerate}
\item[$(\mathrm{a})$] $(P,\cdot)$ is a semisymmetric quasigroup \, ;
\item[$(\mathrm{b})$] $(P,\alpha)$ is a diagonal algebra;
\item[$(\mathrm{c})$] For all $x_i$ and $y_j$ in $P$, the \emph{compatibility identity}
\begin{equation}\label{E:cmptblty}
(x_1y_1,x_2y_2,x_3y_3)\alpha =
(x_3,x_1,x_2)\alpha \cdot (y_2,y_3,y_1)\alpha
\end{equation}
is satisfied.
\end{enumerate}
\end{definition}

By Definition~\ref{D:semisymd}(a), the reduct $(P,\cdot)$ within a semisymmetrized algebra $(P,\cdot,\alpha)$ is a semisymmetric quasigroup. By Definition~\ref{D:semisymd}(b), the reduct $(P,\alpha)$ is a diagonal algebra.

\begin{theorem}\label{T:semisymd}\cite[Th.~35]{SmQHSRA}
The category $\mathbf{Qtp}$ of quasigroup homotopies is equivalent to the variety of semisymmetrized algebras.
\end{theorem}

\subsection{Linearization of semisymmetrized algebras}

Linearization is a useful general technique for the analysis of algebraic structures (compare \cite[Ch.~7]{RS}, for example). Here, we will linearize semisymmetrized algebras.

\subsubsection{Semisymmetric quasigroup structure}

Consider an abelian group $P$ with an automorphism $\rho$ and multiplication
\begin{equation}\label{E:rhomultP}
x\cdot y=x\rho+y\rho^{-1}
\end{equation}
such that $(P,\cdot)$ is a semisymmetric quasigroup. Note that
$$
\rho=R(0)=L(0)^{-1}
$$
in accordance with \eqref{E:R(x)L(x)}. Then
\begin{align*}
y
&
=
(x.y)\cdot x
=
(x\rho+y\rho^{-1})\rho+x\rho^{-1}
=x(\rho^2+\rho^{-1})+y \, .
\end{align*}
Equating coefficients of $x$ gives $\rho^3=-1$. Conversely, the multiplication 
\eqref{E:rhomultP} gives a semisymmetric quasigroup structure $(P,\cdot)$ if $\rho^3=-1$.

\subsubsection{Diagonal algebra structure}

Now consider the ternary operation
$$
(x_1,x_2,x_3)\alpha=x_1a_1+x_2a_2+x_3a_3
$$
with endomorphisms $a_1,a_2,a_3$ of $P$. Idempotence is equivalent to
$$
a_1+a_2+a_3=1 \, .
$$
For the diagonal identity, consider
\begin{align*}
x_{11}a_1+&x_{22}a_2+x_{33}a_3
=
(x_{11},x_{22},x_{33})\alpha
\\
&
=
\big(
(x_{11},x_{12},x_{13})\alpha,
(x_{21},x_{22},x_{23})\alpha,
(x_{31},x_{32},x_{33})\alpha
\big)
\alpha
\\
&
=
x_{11}a_1^2+x_{12}a_2a_1+x_{13}a_3a_1
\\
&
\rule{0mm}{0mm}
+x_{21}a_2a_1+x_{22}a_2^2+x_{23}a_3a_2
\\
&
\rule{0mm}{0mm}
+x_{31}a_3a_1+x_{32}a_2a_3+x_{33}a_3^3 \,.
\end{align*}
Equating coefficients shows that the diagonal algebra structure is equivalent to the fact that $\{a_1,a_2,a_3\}$ forms a complete set of orthogonal idempotents.

\subsubsection{Compatibility}

In the linear setting, compatibility amounts to
\begin{align*}
(x_1\rho&+y_1\rho^{-1})a_1
+(x_2\rho+y_2\rho^{-1})a_2
+(x_3\rho+y_3\rho^{-1})a_3
\\
&
=
(x_1y_1,x_2y_2,x_3y_3)\alpha
\\
&
=
(x_3,x_1,x_2)\alpha \cdot (y_2,y_3,y_1)\alpha
\\
&
=
(x_3a_1+x_1a_2+x_2a_3)\rho
+
(y_2a_1+y_3a_2+y_1a_3)\rho^{-1} \, .
\end{align*}
Equating coefficients of $x_1,x_2,x_3,y_1,y_2,y_3$ yields
\begin{align*}
\rho a_1=a_2\rho,\quad
\rho a_2=a_3\rho,\quad
\rho a_3=a_1\rho
\end{align*}
and
\begin{align*}
\rho^{-1}a_1=a_3\rho^{-1},\quad
\rho^{-1}a_2=a_1\rho^{-1},\quad
\rho^{-1}a_3=a_2\rho^{-1} \, ,
\end{align*}
each of which is equivalent to the conjugation relations
\begin{align*}
a_1=a_2^\rho,\quad
a_2=a_3^\rho,\quad
a_3=a_1^\rho
\end{align*}
in the endomorphism ring of the abelian group $P$.

\subsubsection{Summary}\label{SSS:LinSemiS}

The results of this section may be collected as follows.

\begin{theorem}\label{T:LinSemiS}
Let $P$ be an abelian group. Consider operations
\begin{equation}\label{E:dotononP}
x\cdot y=x\rho+y\lambda
\end{equation}
and
\begin{equation}\label{E:alphaonP}
(x_1,x_2,x_3)\alpha=x_1a_1+x_2a_2+x_3a_3
\end{equation}
on the set $P$, with endomorphisms $\rho,\lambda,a_1,a_2,a_3$ of the abelian group $P$. Then $(P,\cdot,\alpha)$ is a semisymmetrized algebra if and only if $\{a_1,a_2,a_3\}$ forms a complete set of orthogonal idempotents,
$
\rho\lambda=\lambda\rho=1_P=-\rho^3
$,
and the conjugation relations
$
a_1=a_2^\rho
$,
$
a_2=a_3^\rho
$,
$
a_3=a_1^\rho
$
hold in the endomorphism ring of the abelian group $P$.
\end{theorem}

\subsection{Semisymmetrized algebras and semisymmetrizations}\label{SS:SeSyAlSe}

In \cite[Th.~33]{SmQHSRA}, it was shown in the general case, by use of abstract categorical methods, that the semisymmetric quasigroup reducts of semisymmetrized algebras are (isomorphic to) semisymmetrizations. This result will now be established concretely in the linear case described by Theorem~\ref{T:LinSemiS}.

Consider an abelian group $P$ with endomorphisms $\rho,\lambda,a_1,a_2,a_3$ satisfying the conditions of Theorem~\ref{T:LinSemiS}. It will be shown that   
$(P,\cdot)$ has the structure of the semisymmetrization of an abelian group isotope, as described in \cite[\S2]{ImKoSm}.

Since the endomorphisms $a_1,a_2,a_3$ form a complete set of orthogonal idempotents (or in linear algebra terminology: projections), the abelian group $P$ admits a biproduct decomposition
$
P=S_1\oplus S_2\oplus S_3
$
with $S_i=\mathrm{Im}\,a_i$ for $1\le i\le 3$. We will write elements of $P$ as triples $[x_1\ x_2\ x_3]$ with $x_i\in S_i$ for $1\le i\le 3$.

\begin{lemma}\label{L:SeSyAlSe}
There are abelian group homomorphisms
\begin{equation}\label{E:theta123}
\theta_1\colon S_1\to S_3\,,\quad
\theta_2\colon S_2\to S_1\,,\quad
\theta_3\colon S_3\to S_2
\end{equation}
such that
\begin{align*}
[x_1\ 0\ 0]\rho=[0\ 0\ x_1\theta_1]\,,\
[0\ x_2\ 0]\rho=[x_2\theta_2\ 0\ 0]\,,\
[0\ 0\ x_3]\rho=[0\ x_3\theta_3\ 0]
\end{align*}
with $x_i\in P_i$ for $1\le i\le 3$.
\end{lemma}

\begin{proof}
Since $a_1$ is a projection, we have
\begin{align*}
[x_1\ 0\ 0]\rho
=[x_1\ 0\ 0]a_1\rho
=[x_1\ 0\ 0]\rho a_3
\end{align*}
by the conjugation relation $a_3=a_1^\rho$ of Theorem~\ref{T:LinSemiS}. Now $[x_1\ 0\ 0]\rho a_3$ has the form $[0\ 0\ x_1\theta_1]$ for some homomorphism $\theta_1\colon S_1\to S_3$, as required. The other relations are obtained in similar fashion.
\end{proof}

By Lemma~\ref{L:SeSyAlSe}, it follows that
$$
[x_1\ x_2\ x_3]\rho
=
[x_1\ x_2\ x_3]
\begin{bmatrix}
0 &0 &\theta_1\\
\theta_2 &0 &0\\
0 &\theta_3 &0
\end{bmatrix}
$$
for $[x_1\ x_2\ x_3]$ in $P$. The relation $\rho^3=-1$ of Theorem~\ref{T:LinSemiS} then implies that
\begin{equation}\label{E:theta123}
\theta_1\theta_3\theta_2=-1_{S_1}\,,\quad
\theta_2\theta_1\theta_3=-1_{S_2}\,,\quad
\theta_3\theta_2\theta_1=-1_{S_3}\,,
\end{equation}
whence the homomorphisms \eqref{E:theta123} of Lemma~\ref{L:SeSyAlSe} are isomorphisms. Thus
\begin{equation}\label{E:AiS1S2S3}
A\cong S_1\cong S_2\cong S_3
\end{equation}
for some abelian group $A$.

Now for elements $\mathbf x=[x_1\ x_2\ x_3]$ and $\mathbf y=[y_1\ y_2\ y_3]$ of $P$, we have
\begin{align*}
\mathbf x\cdot\mathbf y
&
=
[x_1\ x_2\ x_3]
\begin{bmatrix}
0 &0 &\theta_1\\
\theta_2 &0 &0\\
0 &\theta_3 &0
\end{bmatrix}
+
[y_1\ y_2\ y_3]
\begin{bmatrix}
0 &0 &\theta_1\\
\theta_2 &0 &0\\
0 &\theta_3 &0
\end{bmatrix}^{-1}
\\
&
=
[x_1\ x_2\ x_3]
\begin{bmatrix}
0 &0 &\theta_1\\
\theta_2 &0 &0\\
0 &\theta_3 &0
\end{bmatrix}
+
[y_1\ y_2\ y_3]
\begin{bmatrix}
0 &\theta_2^{-1} &0\\
0 &0 &\theta_3^{-1}\\
\theta_1^{-1} &0 &0
\end{bmatrix}
\\
&
=
\begin{bmatrix}
x_2\theta_2+y_3\theta_1^{-1}
&
x_3\theta_3+y_1\theta_2^{-1}
&
x_1\theta_1+y_2\theta_3^{-1}
\end{bmatrix}
\end{align*}
Comparison with \eqref{E:semisym} yields:
\begin{align}\label{E:ratlinSS}
\begin{cases}
y_3/x_2&=y_3\theta_1^{-1}+x_2\theta_2\in S_1\, ,
\\
y_1\backslash x_3&=y_1\theta_2^{-1}+x_3\theta_3\in S_2\, ,
\\
x_1\cdot y_2&=x_1\theta_1+y_2\theta_3^{-1}\in S_3
\end{cases}
\end{align}
for $x_i,y_i\in S_i$.

\begin{theorem}\label{T:LiSSARvA}
Consider an abelian group $P$ equipped with endomorphisms $\rho,\lambda,a_1,a_2,a_3$ satisfying the conditions of Theorem~\ref{T:LinSemiS}. Define $S_i=\mathrm{Im}\,a_i$ for $1\le i\le 3$, with the abelian group isomorphisms $\theta_i$ of \eqref{E:theta123}.
\begin{enumerate}
\item[$(\mathrm a)$] The operations \eqref{E:ratlinSS} define a reversible automaton on the triple $(S_1,S_2,S_3)$ of state spaces.
\item[$(\mathrm b)$] The semisymmetrized algebra $(P,\cdot,\alpha)$ with \eqref{E:dotononP}, \eqref{E:alphaonP} is isomorphic to the semisymmetrization of (a quasigroup isotopic to) the abelian group $A$ of \eqref{E:AiS1S2S3}.
\end{enumerate}
\end{theorem}

\begin{proof}
(a) To verify (ILA) of \eqref{E:revautid}, consider $x_1\in S_1$ and $x_2\in S_2$. Then
\begin{align*}
x_1\backslash (x_1\cdot x_2)
&
=
x_1\theta_2^{-1}+\big(x_1\theta_1+x_2\theta_3^{-1}\big)\theta_3
\\
&
=
x_1\theta_2^{-1}+x_1\theta_1\theta_3+x_2=x_2\, ,
\end{align*}
since $x_1\theta_1\theta_3=-x_1\theta_2^{-1}$ by the first equation of \eqref{E:theta123}. Verification of the remaining parts of \eqref{E:revautid} is similar.
\vskip 3mm
\noindent
(b) We will apply Lemma~\ref{L:ReAt2Qgp} to the reversible automaton of (a). Choose the set isomorphisms \eqref{E:isoq2aut} according to the commutative diagram
$$
\xymatrix{
S_1
\ar[r]^{\theta_2^{-1}}
&
S_2
\ar[r]^{\theta_3^{-1}}
&
S_3
\\
&
A
\ar[ul]^{l_1=1_A}
\ar[u]^{l_2}
\ar[ur]_{l_3}
}
$$
so that $l_2=\theta_2^{-1}$ and $l_3=l_2\theta_3^{-1}=\theta_2^{-1}\theta_3^{-1}=\big(\theta_3\theta_2\big)^{-1}=-\theta_1^{-1}$, the last equation here holding by the last equation of \eqref{E:theta123}. Then from the last equation of \eqref{E:ratlinSS}, and \eqref{E:dotonq} in Lemma~\ref{L:ReAt2Qgp}, the semisymmetrized algebra $(P,\cdot,\alpha)$ is the semisymmetrization of the operation
\begin{align*}
x\cdot y
&
=
\big(x\theta_1+yl_2\theta_3^{-1}\big)l_3^{-1}
=
(x\theta_1+yl_3)l_3^{-1}
=
x\theta_1l_3^{-1}+y
=-x+y
\end{align*}
of opposed subtraction on elements $x,y$ of the abelian group $A$.
\end{proof}

\section{Words in free central quasigroups}

\subsection{Representing free central piques}\label{SS:CePiFrGp}

Recall that a \emph{pique} is a quasigroup with a pointed idempotent element. An algebra $Q$ is \emph{central} if the diagonal $\widehat{Q}=\{(x,x)\mid x\in Q\}$ is a normal subalgebra (congruence class) in the direct square $Q\times Q$.

\begin{proposition}\cite[\S3.2]{RTIGFQ},\cite[Prop.~11.1]{IQTR}\label{P:CePiFrGp}
Let $\langle R,L\rangle$ be the free group on the two-element set $\{R,L\}$. Then central piques are equivalent to right modules over the integral group algebra $\mathbb Z\langle R,L\rangle$, with
\begin{align}\label{E:centmult}
x\cdot y&=xR+yL\,,\\ \label{E:centrtdv}
x/y&=xR^{-1}-yLR^{-1}\,,\\ \label{E:centltdv}
x\backslash y&=-xRL^{-1}+yL^{-1}\,,\\
e&=0
\end{align}
as the multiplication, right division, left division, and pointed idempotent.
\end{proposition}

\begin{corollary}\label{C:CePiFrGp}
The free central pique over an alphabet $A$ is represented faithfully by the subpique of $\big(\mathbb Z\langle R,L\rangle A,\cdot,/,\backslash,0\big)$ generated by the subset $A$ of the free right $\mathbb Z\langle R,L\rangle$-module $\mathbb Z\langle R,L\rangle A$.
\end{corollary}

Although the models of free central piques given by Corollary~\ref{C:CePiFrGp} are faithful, they do have a serious disadvantage when it comes to studying triality symmetry: the polynomials in $R^{\pm1},L^{\pm1}$ giving the divisions are not homogeneous. This complicates any attempt to enumerate free central pique words by degrees of the representing polynomials.

\subsection{The homogeneous representation}

A homogeneous representation of central pique words is initiated by taking the operations as defined in \eqref{E:ratlinSS}, recalling the relationship \eqref{E:theta123}. However, we do make substitutions
$$
\theta_1\mapsto X_2\,\quad
\theta_2\mapsto X_3\,\quad
\theta_3\mapsto X_1
$$
and hence rewrite \eqref{E:theta123} as
\begin{equation}\label{E:cyclerel}
X_3X_2X_1=X_2X_1X_3=X_1X_3X_2=-1 \,.
\end{equation}
We describe the $X_i$ as the \emph{coefficient variables}. The quasigroup operations from Figure~\ref{F:trialops} are displayed in Figure~\ref{F:homogops}. In connection with \eqref{E:cyclerel}, it is worth recordiug the following observation.

\begin{lemma}\label{L:cyclerel}
In \eqref{E:cyclerel}, suppose that the $X_i$ are invertible. Then the equation $X_iX_{i-1}X_{i-2}=-1$ for any one of the cyclic products $X_iX_{i-1}X_{i-2}=-1$ (with indices taken modulo $3$) is equivalent to the equation of each of the other two cyclic products with $-1$.
\end{lemma}

\vskip .2in
\begin{figure}[bht]
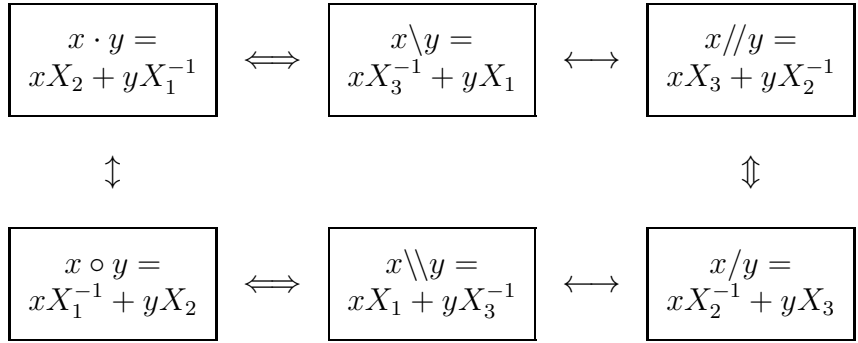

$$
\begin{matrix}
\boxed{\phantom{\Bigg|}
\begin{matrix}
x\cdot y=\\
xX_2+yX_1^{-1}
\end{matrix}
\phantom{\Bigg|}}
&\Longleftrightarrow
&\boxed{\phantom{\Bigg|}
\begin{matrix}
x\backslash y=\\
xX_3^{-1}+yX_1
\end{matrix}
\phantom{\Bigg|}}
&\longleftrightarrow
&\boxed{\phantom{\Bigg|}
\begin{matrix}
x/\negthinspace/y=\\
xX_3+yX_2^{-1}
\end{matrix}
\phantom{\Bigg|}}
\\
&            &       &                   & \\
\updownarrow   &            &       &                   &\Updownarrow   \\
&            &       &                   & \\
\boxed{\phantom{\Bigg|}
\begin{matrix}
x\circ y=\\
xX_1^{-1}+yX_2
\end{matrix}
\phantom{\Bigg|}}
&\Longleftrightarrow
&\boxed{\phantom{\Bigg|}
\begin{matrix}
x\backslash\negthinspace\backslash y=\\
xX_1+yX_3^{-1}
\end{matrix}
\phantom{\Bigg|}}
&\longleftrightarrow
&\boxed{\phantom{\Bigg|}
\begin{matrix}
x/y=\\
xX_2^{-1}+yX_3
\end{matrix}
\phantom{\Bigg|}}
\end{matrix}
$$
\vskip .3in
\caption{Homogeneous linear representation of quasigroup operations.
}\label{F:homogops}
\end{figure}
In Figure~\ref{F:trialops}, $\leftrightarrow$ stood for right multiplication by $\sigma$, while $\Leftrightarrow$ stood for right multiplication by $\tau$. Now, in Figure~\ref{F:homogops}, $\leftrightarrow$ means inverting the coefficient variables and applying $(1\ 2)$ to their suffices. Similarly, $\Leftrightarrow$ means inverting the coefficient variables and applying $(2\ 3)$ to their suffices. These actions may be displayed explicitly as
\begin{align*}
(1\ 2)\colon
X_1\mapsto X_2^{-1}, \
X_2\mapsto X_1^{-1}, \
X_3\mapsto X_3^{-1}\phantom{\,.}
\end{align*}
and
\begin{align*}
(2\ 3)\colon
X_1\mapsto X_1^{-1}, \
X_2\mapsto X_3^{-1}, \
X_3\mapsto X_2^{-1}\, .
\end{align*}
The relations \eqref{E:cyclerel} are invariant under these actions, demonstrating the homogeneity of the representation under the triality group $S_3$.

\subsection{Faithfulness of the homogeneous representation}

We now show that the free versions of the homogeneous representations discussed in the previous section are faithful. Equating the coefficients from the top row of Figure~\ref{F:homogops} with the coefficients from \eqref{E:centmult}--\eqref{E:centltdv} yields
\begin{equation}\label{E:XifromLR}
X_1=L^{-1}\,,\quad
X_2=R\,,\quad
X_3=-LR^{-1}
\end{equation}
with
\begin{equation}\label{E:XiLRcy-1}
X_3X_2X_1=-LR^{-1}RL^{-1}=-1 \,.
\end{equation}
By \eqref{E:XifromLR}, the $X_i$ here are invertible. Then by Lemma~\ref{L:cyclerel}, the relation \eqref{E:XiLRcy-1} implies all the equations of \eqref{E:cyclerel}. Thus the homogeneous representation by coefficient variables is capable of modeling the faithful representation from \S\ref{SS:CePiFrGp}. We may summarize the preceding observations and refomulate Proposition~\ref{P:CePiFrGp} as follows.

\begin{theorem}[Homogeneous representation of central piques]\label{T:CePiFrGp}
Consider the algebra
\begin{equation}\label{E:PresAlgS}
S=
\mathbb Z\langle X_1,X_2,X_3\mid 1+X_3X_2X_1,1+X_2X_1X_3,1+X_1X_3X_2\rangle
\end{equation}
of polynomials, with integral coefficients, in the non-commuting coefficient variables $X_1,X_2,X_3$, subject to the annihilation of the indicated polynomials. 
\begin{enumerate}
\item[$(\mathrm a)$]
The actions
\begin{align*}
(1\ 2)\colon
X_1\mapsto X_2^{-1}, \
X_2\mapsto X_1^{-1}, \
X_3\mapsto X_3^{-1}\phantom{\,.}
\end{align*}
and
\begin{align*}
(2\ 3)\colon
X_1\mapsto X_1^{-1}, \
X_2\mapsto X_3^{-1}, \
X_3\mapsto X_2^{-1}\, .
\end{align*}
generate a group $S_3$ of automorphisms of $S$.
\item[$(\mathrm b)$]
Central piques are equivalent to right $S$-modules, with
\begin{align}\label{E:ccntmult}
x\cdot y&=xX_2+yX_1^{-1}\,,\\ \label{E:ccntrtdv}
x/y&=xX_2^{-1}+yX_3\,,\\ \label{E:ccntltdv}
x\backslash y&=xX_3^{-1}+yX_1\,,\\
e&=0
\end{align}
as the respective multiplication, right division, left division, and pointed idempotent.
\item[$(\mathrm c)$]
The action of $S_3$ on $S$ from $(\mathrm a)$ induces a triality action on central piques as represented in $(\mathrm b)$.
\end{enumerate}
\end{theorem}

\begin{remark}\label{R:nonegtiv}
(a)
Abstractly, \eqref{E:PresAlgS} is an alternative presentation of the integral free group algebra $\mathbb Z\langle R,L\rangle$ from Proposition~\ref{P:CePiFrGp}.
\vskip 2mm
\noindent
(b)
In comparison with Proposition~\ref{P:CePiFrGp}, note that the integral coefficients of the monomials appearing in central quasigroup words represented by Theorem~\ref{T:CePiFrGp} are all nonnegative.
\end{remark}

\subsection{Representing free quasigroup words}

Consider
\begin{equation}\label{E:TwoWords}
((a_0/a_1).a_2a_3)/(a_4\backslash a_0)
=(a_4.a_2a_3)/a_1
\end{equation}
as an example of an equality between distinct reduced free quasigroup words (in the sense of \cite{JSW}) that holds in the free central quasigroup on the alphabet $\{a_0,\dots,a_4\}$. With the homogeneous representation given by Theorem~\ref{T:CePiFrGp}, the representation of $((a_0/a_1).a_2a_3)/(a_4\backslash a_0)$ is 
\begin{equation}\label{E:len6word}
a_0X_2^{-1}+a_1X_3+a_2X_2X_1^{-1}X_2^{-1}
+a_3X_1^{-2}X_2^{-1}+a_4+a_0X_1X_3\,,
\end{equation}
while the representation for $(a_4.a_2a_3)/a_1$ is
\begin{equation}\label{E:len4word}
a_1X_3+a_2X_2X_1^{-1}X_2^{-1}+a_3X_1^{-2}X_2^{-1}+a_4\,.
\end{equation}
Since the first term in~\ref{E:len6word} is equivalent to $-a_0X_1X_3$, the homogeneous representations of the two words are equal.

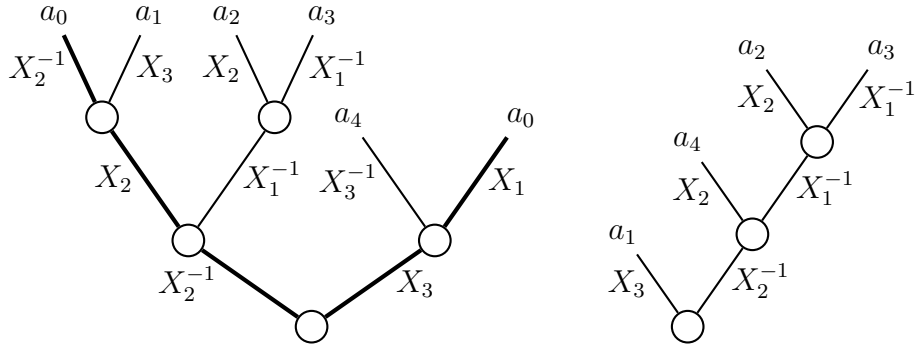
\begin{figure}[bht]
\centering
\begin{tikzpicture}
\node (1) at (0,0) [knode] {};
\node (2) at ($(1)+(35:2)$) [knode]{};
\node (3) at ($(1)+(145:2)$) [knode] {};
\node (4) at ($(3)+(55:2)$) [knode] {};
\node (5) at ($(3)+(125:2)$) [knode] {};
\node (6) at ($(5)+(65:1.5)$)  {$a_1$};
\node (7) at ($(5)+(115:1.5)$)  {$a_0$};;
\node (8) at ($(4)+(65:1.5)$)  {$a_3$};
\node (9) at ($(4)+(115:1.5)$)  {$a_2$};
\node (10) at ($(2)+(55:2)$)  {$a_0$};
\node (11) at ($(2)+(125:2)$)  {$a_4$};
\draw[ultra thick] (1) edge node[right] {$\,\,X_3$} (2);
\draw[ultra thick] (1) edge node[left] {$X_2^{-1}\,\,\,\,$} (3);
\draw[thick] (3) edge node[right] {$X_1^{-1}$} (4);
\draw[ultra thick] (3) edge node[left] {$X_2$} (5);
\draw[thick] (5) edge node[right] {$X_3$} (6);
\draw[ultra thick] (5) edge node[left] {$X_2^{-1}$} (7);
\draw[thick] (4) edge node[right] {$X_1^{-1}$} (8);
\draw[thick] (4) edge node[left] {$X_2$} (9);
\draw[ultra thick] (2) edge node[right] {$X_1$} (10);
\draw[thick] (2) edge node[left] {$X_3^{-1}\,$} (11);
\begin{scope}[xshift=5cm]
\node (1) at (0,0) [knode] {};
\node (2) at ($(1)+(55:1.5)$) [knode]{};
\node (3) at ($(1)+(125:1.5)$) {$a_1$};
\node (4) at ($(2)+(55:1.5)$)  [knode]{};
\node (5) at ($(2)+(125:1.5)$)  {$a_4$};
\node (6) at ($(4)+(55:1.5)$)  {$a_3$};
\node (7) at ($(4)+(125:1.5)$)  {$a_2$};
\draw[thick] (1) edge node[right] {$X_2^{-1}$} (2);
\draw[thick] (1) edge node[left]{$X_3$}(3);
\draw[thick] (2) edge node[right]{$X_1^{-1}$} (4);
\draw[thick] (2) edge node[left] {$X_2$} (5);
\draw[thick] (4) edge node[right] {$X_1^{-1}$} (6);
\draw[thick] (4) edge node[left] {$X_2$} (7);
\end{scope}
\end{tikzpicture}
\caption{Two equivalent labelled rooted binary trees.}\label{F:TwoTrees}
\end{figure}

Free quasigroup words are represented by so-called \emph{parsing trees}, rooted binary trees in which generators populate the leaves, and where quasigroup operations are specified at each internal node \cite[\S2.3, \S3.3]{JSW}. With the homogeneous representation of Theorem~\ref{T:CePiFrGp}, the label of an internal node is replaced by labels for the two edges meeting at that node. In accord with Theorem~\ref{T:CePiFrGp}(b), these edges are labelled by the corresponding coefficient variables or their inverses.

The two trees in Figure~\ref{F:TwoTrees} respectively represent the words from \eqref{E:TwoWords}. From each leaf, following the edges down to the root vertex produces a monomial multiple in the homogeneous representation, contributing to the respective representations \ref{E:len6word} and~\ref{E:len4word}. The equivalence of the two words may be observed from Figure~\ref{F:TwoTrees}, noting the collapse of the subtree on the left hand side, identified by the thicker edges, which leads to the elimination of the argument $a_0$.

\subsection{Argument elimination}\label{SS:ArgtElim}

The pattern of collapse observed in Figure~\ref{F:TwoTrees} may be generalized, as displayed in Figure~\ref{F:PatOfCol}. Here, $1\le i\le 3$, and the indices on the coefficient variables are to be interpreted modulo $3$. The symbol $a_r$ denotes a particular letter from the generating alphabet. The dashed lines denote paths, of nonnegative even length, such that the product of the (possibly empty) ordered set of edge labels is $1$.
\begin{figure}[htb]
$$
\xymatrix{
a_r
\ar@{--}[d]
&
&
a_r
\ar@{--}[d]
&
&
&
a_r
\ar@{--}[d]
&
&
a_r
\ar@{--}[d]
\\
\bigcirc
\ar@{-}[ddd]_{X_i^{-1}}
&
&
\bigcirc
\ar@{-}[d]^{X_{i-1}}
&
&
&
\bigcirc
\ar@{-}[ddd]_{X_i}
&
&
\bigcirc
\ar@{-}[d]^{X_{i+1}^{-1}}
\\
&
&
\bigcirc
\ar@{--}[d]
&
&
&
&
&
\bigcirc
\ar@{--}[d]
\\
&
&
\bigcirc
\ar@{-}[d]^{X_{i+1}}
&
&
&
&
&
\bigcirc
\ar@{-}[d]^{X_{i-1}^{-1}}
\\
\bigcirc
\ar@{--}[dr]
&
&
\bigcirc
\ar@{--}[dl]
&
&
&
\bigcirc
\ar@{--}[dr]
&
&
\bigcirc
\ar@{--}[dl]
\\
&
\bigcirc
&
&
&
&
&
\bigcirc
&
}
$$
\caption{The patterns of argument elimination.}\label{F:PatOfCol}
\end{figure}
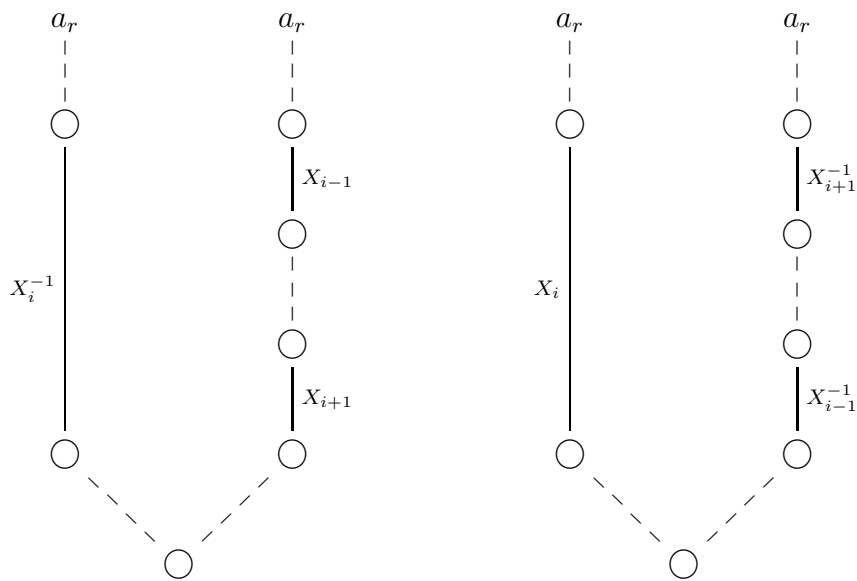
The thicker lines in Figure~\ref{F:TwoTrees} provide an instance of the left-hand pattern in Figure~\ref{F:PatOfCol}, with $i=2$ and $r=0$. All three dashed paths to the right of the root have zero length, while either one of the dashed paths to the left of the root may be taken to have zero length, the other dashed path to the left then having length two.

\subsection{A matrix implementation}\label{SS:16shrtst}

As an illustration, consider the matrix representation
\begin{equation}\label{E:XiinSL2Z}
X_1=
\begin{bmatrix}
1 &0\\
-\sqrt 5 &1
\end{bmatrix}\,,\quad
X_2=
\begin{bmatrix}
1 &\sqrt 5\\
0 &1
\end{bmatrix}\,,\quad
X_3=
\begin{bmatrix}
-1 &\sqrt 5\\
-\sqrt 5 &4
\end{bmatrix}\,.
\end{equation}
The matrices \eqref{E:XiinSL2Z} satisfy
$
X_3X_2X_1=-1\,.
$
By Lemma~\ref{L:cyclerel}, it follows that they satisfy all the relations of \eqref{E:cyclerel}. Consider the central pique $(\mathbb R^2,\cdot,[0\ 0])$ with the operations of Theorem~\ref{T:CePiFrGp}. The sixteen shortest words in the central subquasigroup generated by $a=[\sqrt 2\quad 2\,]$ are displayed in Figure~\ref{F:16shrtst}. Exact triality symmetry of the words is reflected within the figure as an approximate geometric symmetry.

\begin{figure}[thb]
\includegraphics[width=140mm,height=120mm]{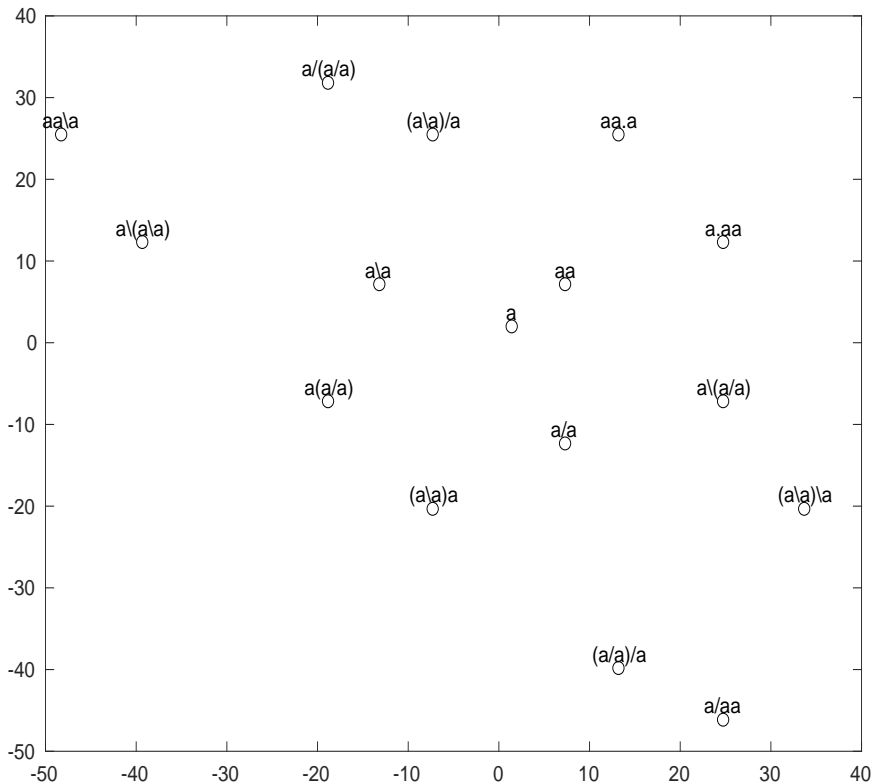}
\caption{The sixteen shortest words in a central quasigroup.}\label{F:16shrtst}
\end{figure}


\end{document}